\documentclass[a4paper,10pt]{article}
\usepackage[utf8x]{inputenc}
\usepackage[namelimits]{amsmath}
\usepackage{amsxtra}
\usepackage[varg]{txfonts}
\usepackage[unicode]{hyperref}
\usepackage[amsmath,hyperref,framed,thmmarks]{ntheorem}
\usepackage{mathrsfs}
\usepackage{graphicx}

\theoremseparator{.}
\makeatletter
\renewcommand{\newframedtheorem}[1]{%
\theoremprework{\framed\vspace{-1.5ex}}%
\theorempostwork{\vspace{-2ex}\endframed}%
\newtheorem@i{#1}%
}
\makeatother
\newtheorem{theorem}{Theorem}[section]
\newtheorem{proposition}[theorem]{Proposition}
\newtheorem{lemma}[theorem]{Lemma}

\theorembodyfont{\upshape}
\theoremsymbol{$\Diamond$}

\makeatletter
\newtheoremstyle{proof}%
{\item[\theorem@headerfont\hskip\labelsep ##1\theorem@separator]}%
{\item[\ifx\@empty##3\else\theorem@headerfont\hskip \labelsep ##1\ of\ ##3\theorem@separator\fi]}
\makeatother
\theoremheaderfont{\itshape}
\theoremstyle{proof}
\theoremseparator{.}
\theoremsymbol{\ensuremath{\Box}}
\newtheorem{proof}{Proof}
\theoremheaderfont{\sffamily\bfseries}\theorembodyfont{\sffamily\upshape}
\theoremstyle{plain}
\theoremnumbering{alph}
\theoremsymbol{$\Diamond$}
\theoremprework{\small}
\theorempostwork{\normalsize}

\hypersetup{
breaklinks=true,
colorlinks=true,
linkcolor=blue,
citecolor=blue,
urlcolor=blue,
}

\def\bfseries{\fontseries \bfdefault \selectfont \boldmath}

\makeatletter
\def\@fnsymbol#1{\ensuremath{\ifcase#1\or *\or **\or {}\or {}\or {**}{**}\else\@ctrerr\fi}}
\makeatother

\newcommand{\abs}[1]{\ensuremath{\left|#1\right|}}
\newcommand{\br}[1]{\ensuremath{\left(#1\right)}}

\renewcommand{\d}{\ensuremath{\,\mathrm{d}}}

\newcommand{\norm}[1]{\ensuremath{\left\|#1\right\|}}
\newcommand{\seminorm}[1]{\ensuremath{\left[#1\right]}}
\newcommand{\seq}[1]{\ensuremath{\br{#1}_{\eps>0}}}

\newcommand{\set}[1]{\ensuremath{\left\{#1\right\}}}

\renewcommand{\sp}[1]{\ensuremath{\left\langle#1\right\rangle}}

\numberwithin{equation}{section}

\renewcommand{\a}{\ensuremath{\alpha}}
\renewcommand{\C}{\ensuremath{\mathbb{C}}}

\DeclareMathOperator{\lip}{lip}
\newcommand{\D}[1]{{#1}(u+w)-{#1}(u)}
\newcommand{\dg}{{\gamma}'}
\newcommand{\dgt}{\gt'}
\renewcommand{\dh}{{h}'}
\newcommand{\dtg}{{\tilde{\gamma}}'}

\newcommand{\E}{\ensuremath{E^{(\a)}}}
\newcommand{\Ee}[1][\eps]{\ensuremath{E_{#1}^{(\a)}}}
\newcommand{\eps}{\ensuremath{\varepsilon}}

\newcommand{\g}{\gamma}
\renewcommand{\G}[1][\g]{\ensuremath{\mathcal G_{{#1}}^{(\a)}}}

\newcommand{\Lp}[1][p]{\ensuremath{L^{\scriptstyle #1}}}
\newcommand{\length}{\ensuremath{\mathscr L}}
\newcommand{\N}{\ensuremath{\mathbb{N}}}

\newcommand{\Q}[1][\a]{Q^{(#1)}}
\newcommand{\R}{\ensuremath{\mathbb{R}}}
\newcommand{\refeq}[2][=]{\ensuremath{\stackrel{\text{\makebox[0cm][c]{\eqref{eq:#2}}}}{#1}}}
\renewcommand{\rho}{\ensuremath{\varrho}}
\newcommand{\RR}[1][\g]{R^{(\a)}}
\newcommand{\rzd}{\ensuremath{(\R/\Z,\R^n)}}
\newcommand{\s}{\ensuremath{\sigma}}

\newcommand{\tQ}[1][\a]{\widetilde{Q}^{(#1)}_\lambda}
\newcommand{\ts}{\ensuremath{{\tilde{\sigma}}}}

\newcommand{\W}[1][(\a+1)/2,2]{\ensuremath{H^{\scriptstyle #1}}}
\newcommand{\Wir}[1][(\a+1)/2,2]{\ensuremath{H_{\mathrm{ir}}^{#1}}}
\newcommand{\Z}{\ensuremath{\mathbb{Z}}}

\title{Stationary Points of O'Hara's Knot Energies}
\author{%
 Simon Blatt\thanks{Mathematics Institute, Zeeman Building, University of Warwick, Coventry CV4 7AL, United Kingdom, \url{S.Blatt@warwick.ac.uk}} \and
 Philipp Reiter\thanks{Abteilung f\"ur Angewandte Mathematik, Universit\"at Freiburg, Hermann-Herder-Stra\ss e 10, 79104 Freiburg i.~Br., Germany, \url{reiter@mathematik.uni-freiburg.de}}%
 \footnote{MSC (2000) 42A45, 53A04, 57M25}%
 \footnote{The final publication is available at \url{www.springerlink.com}.}}

\begin{document}
\maketitle


\begin{abstract}
In this article we study the regularity of stationary points of the knot energies~$\E$
introduced by O'Hara in \cite{oha:en,oha:fam-en,oha:en-func}
in the  range $\a\in(2,3)$.
In a first step we prove that $\E$ is $C^1$
on the set of all regular embedded curves belonging to  $\W\rzd$
and calculate its derivative.
After that we use the structure of the Euler-Lagrange equation to study the regularity of
stationary points of $\E$ plus a 
positive multiple of the length. We show that stationary points of finite energy are of class
$C^\infty$ --- so especially all local minimizers of~$\E$ among curves with fixed length 
are smooth.
\end{abstract}
\tableofcontents

\section{Introduction}

The motion of a knotted charged fiber within a viscous liquid
served as model for the definition of so-called \emph{knot energies}
introduced by Fukuhara~\cite{fete}.
One hopes that it will reach a stationary point minimizing its electrostatic energy
and that the resulting shape will help to determine its knot type.
The general idea is that this procedure leads to
a ``nicer shape'' for a given knot in the same knot class, i.~e.\@ a
representative that is as little entangled as possible with preferably large distances between
different strands.

For a general definition and
an outline of different knot energies we refer the reader to O'Hara~\cite{oha:en-kn}.
Recent developements include the investigation of geometric curvature energies such as the integral Menger curvature,
see Strzelecki, Szuma{\'n}ska, and von der Mosel~\cite{SM4,SM5},
which also extends to surfaces~\cite{SM6},
or tangent-point energies~\cite{SM7}
whose domains can be characterized via Sobolev-Slobodeckij spaces~\cite{blatt:tpe}.
Attraction phenomena may also be modeled by a corresponding ``inverse
knot energy'', see Alt et~al.~\cite{alt-felix}
for an example from mathematical biology discussing interaction between pairs of
filaments via cross-linkers.

In this paper, we investigate stationary points of the most prominent family of knot energies $\E:C^{0,1}\rzd\to[0,\infty]$,
\begin{equation*}
  \g \longmapsto
  \int_{\R/\Z}\int_{-1/2}^{1/2}
  \br{\frac{1}{\abs{\g(u+w)-\g(u)}^\a} - \frac{1}{d_\g(u+w,u)^\a}}
  \abs{\dg(u+w)}\abs{\dg(u)} \d w\d u, \label{eq:a}
\end{equation*}
where $\a\in[2,3)$, which goes back to O'Hara~\cite{oha:en,oha:fam-en,oha:en-func}.
Here $d_\g(u+w,u)$ denotes the intrinsic distance between $\g(u+w)$ and $\g(u)$ on the curve~$\g$.
More precisely,
$d_\g(u+w,u):= \min\big(\length(\g|_{[u,u+w]}),\length(\g)-\length(\g|_{[u,u+w]})\big)$
provided $\abs w\le\tfrac12$ where
$\length(\g) := \int_0^1 \abs{\dg(\theta)}\d\theta$
is the length of~$\g$.

The energy $E^{(2)}$ was thoroughly studied by Freedman, He, and Wang~\cite{fhw}
who coined the name ``M\"obius energy'' due to the M\"obius invariance of this energy.
While the \emph{existence of mimizers} of the M\"obuis energy is ensured in prime knot classes only,
O'Hara~\cite{oha:fam-en,oha:en-func} proved the existence of minimizers within
\emph{any} knot class if $\a\in(2,3)$.
Abrams et~al.~\cite{acfgh} proved that circles are the global minimizers of all these energies 
among all curves.

As to the \emph{regularity of stationary points}, the first result was obtained by He~\cite{he:elghf} for $\a=2$
who initially assuming $H^{2,3}$-regularity obtained $C^\infty$ by a bootstrapping argument.
Together with a purely geometric result by Freedman, He, and Wang~\cite{fhw} heavily relying on the M\"obius invariance, this gives
$C^\infty$-regularity for \emph{all} local $E^{(2)}$-\emph{minimizers} (which exist at least in prime knot classes).
An outline is given in~\cite{reiter:rtm}.
Moreover, He was able to show that under suitable conditions any
planar (i.~e.\@ $n=2$) stationary point of $E^{(2)}$ is a circle
\cite[Thm.~6.3]{he:elghf}. The argument highly relies on the M\"obius
invariance of $E^{(2)}$."

In~\cite{reiter:rkepdc}, parts of these results were carried over to the energies $\E$ for $\alpha \in [2,3)$.
It was shown that stationary points in $H^{\a,2}\cap H^{2,3}$ 
of the energy $\E + \lambda \mathscr L$, where $\mathscr L$ denotes the length functional and $\lambda>0$ is
a constant, are smooth. Here $H^{s,p}$ denote the \emph{Bessel potential spaces}.
Unfortunately, one does not know whether local minimizers
of $\E$ belong to $H^{\a,2}\cap H^{2,3}$ since the techniques used 
by Freedman, He, and Wang~\cite{fhw} to show the 
regularity of local minimizers completely break down in these cases.

In this article we will close this gap by proving a much stronger
result. We will extend the results in~\cite{reiter:rkepdc} and~\cite{oha:en-func}
and prove smoothness of stationary points of the 
functionals $\E + \lambda\mathscr L$ under very natural conditions: We 
will only assume that the curve $\gamma$ we are looking at is parametrized by arc-length
(which means no loss of generality as $\E$ is invariant of parametrization)
and satisfies $\E(\gamma)<\infty$.

The first step to show this result is to extract as much information regarding the regularity of
$\gamma$ out of the finiteness
of the energy $\E$ as possible.
After some partial result~\cite{blatt-reiter} in this
direction, in~\cite{blatt:bre} a classification of all curves with finite energy was given:
An embedded curve parametrized by arc-length
has finite energy $\E$ if and only if it belongs to the fractional Sobolev space $\W$.

Since formulas for the first variation of $\E$ are only known 
under the assumption that $\gamma \in H^2$, we then have to extend these
to injective curves in $H^{(1+\alpha)/2}$ parametrized by arc-length. In fact our 
method even allows us to show that $\E$ is continuously differentiable on this space.
To state the result,
let
\begin{equation*}
 \textstyle U_\eps := \R/\Z \times \br{[-\frac12,-\eps]\cup[\eps,\frac12]}.
\end{equation*}

\begin{theorem} \label{thm:FrechetDifferential}
  Let $\alpha \in (2,3)$. The energies $\E$ are $C^1$-differentiable on the space of all
  injective regular curves $\gamma \in H^{(1+\alpha)/2}$. Furthermore, if $\gamma$
  is parametrized by arc-length, the derivative 
  at $\gamma$ in direction $h$ is given by\footnote{%
  If $\gamma$ belongs to $ H^{\alpha +1}$ one can use partial
  integration to obtain a formula for the $L^2$ gradient like in~\cite[Thm.~2.24]{reiter:rkepdc}.}%
  \begin{align*}
   \delta\E(\gamma;h) = \lim_{\varepsilon \searrow 0}  \iint_{U_\eps} \Bigg( &
    (\a-2) \frac{\sp{\dg(u),\dh(u)}}{\abs w^\a}
    + 2 \frac{\sp{\dg(u),\dh(u)}}{\abs{\D\g}^\a} \\
    &{}-\a \frac{\sp{\D\g,\D h}}{\abs{\D\g}^{\a+2}}
    \Bigg) \d w\d u.
  \end{align*}
\end{theorem}

Note that this is a principle value integral, i.~e.\@ we may not replace $U_\eps$ by $U_0$.

Now we are in the position to state the main result of this article.

\begin{theorem} \label{thm:rsp}
 Let $\alpha \in (2,3)$  and $\gamma\in C^{0,1}(\mathbb R / \mathbb Z, \mathbb R^n)$ be a curve
 parametrized by arc-length with $\E(\gamma)< \infty$. If $\gamma$ is furthermore a stationary
 point of $\E + \lambda \mathscr L$, i.~e.\@ if
 \begin{equation*}
  \delta \E(\gamma;h) + \lambda \int_{\mathbb R / \mathbb Z} \sp{\gamma',h'} =0 \quad \forall h \in \W[(1+\alpha)/2]
  (\mathbb R / \mathbb Z, \mathbb R^n),
 \end{equation*}
 then $\gamma \in C^\infty(\mathbb R / \mathbb Z, \mathbb R^n)$.
\end{theorem}

The gradient flow of the M\"obius energy $E^{(2)}$ was first discussed by He~\cite{he:elghf} where 
he states short time existence results for smooth initial data.
In~\cite{blatt:gfm}, the short time existence was proven for all intial data in $C^{2,\beta}$,
$\beta >0$, and first long time existence results for this gradient flow near local minimizers
were derived. 
For a discussion of gradient flow for  $\E+ \lambda\mathscr L$ for positive $\lambda$ and $\a\in(2,3)$
or the gradient flow of $\E$ with respect to fixed length we refer the reader to~\cite{blatt:gfoh}.

The energies $\E$ represent only the one-parameter range $p=1$ of the larger family of knot energies
\begin{equation*}
  E^{\a,p}(\g) :=
  \int_{\R/\Z}\int_{-1/2}^{1/2}
  \br{\frac{1}{\abs{\g(u+w)-\g(u)}^\a} - \frac{1}{d_\g(u+w,u)^\a}}^p
  \abs{\dg(u+w)}\abs{\dg(u)} \d w\d u, \label{eq:a-p}
\end{equation*}
where $\a p\ge2$ and $(\a-2)p<1$, see O'Hara~\cite{oha:fam-en,oha:en-func} and~\cite{blatt:bre,blatt-reiter}.
We do \emph{not} expect that our results or the results for the gradient flow of the energies
carry over to $p>1$ as we expect the 
first variation to be a degenerate elliptic operator in this case.

Let us close this introduction by briefly introducing some notation and the
\emph{Sobolev spaces of fractional order} which are also referred to as \emph{Bessel potential spaces}.
For $s\in\R$ and $p\in[1,\infty]$ let
$H^{s,p} := (\mathrm{id}-\Delta)^{-s/2} L^p$ where $\Delta$ denotes the Laplacian.
There are several equivalent definitions, e.~g.\@ by
interpolation. In case $p=2$, which mainly applies to our situation,
the Bessel potential spaces coincide with the Slobodeckij spaces.
This gives rise to the following fundamental characterization of $H^{s,2}$, $s\in(0,\infty)\setminus\N$.

Let $f\in L^2\rzd$. For $s\in(0,1)$ we define the seminorm
\[ \seminorm{f}_{H^{s,2}} := \br{\int_{\R/\Z}\int_{-1/2}^{1/2} \frac{\abs{f(u+w)-f(u)}^2}{\abs w^{1+2s}} \d w\d u}^{1/2}. \]
Then the Sobolev space $H^{k+s,2}\rzd$, $k\in\N\cup\set0$,
is the set of all functions $H^{k,2}\rzd$ for which the norm
\[ \norm f_{H^{s+k,2}} := \norm f_{H^{k,2}} + \seminorm{f^{(k)}}_{H^{s,2}} \]
is finite.
We will frequently use the embedding
\begin{equation}\label{eq:embedding}
 H^{k+s,2}\rzd\hookrightarrow C^{k+s-1/2}\rzd, \qquad s\in(\tfrac12,1),
\end{equation}
see, e.~g., Taylor~\cite[Chap.~4, Prop.~1.5]{taylor}.
For further information on Sobolev spaces we refer to the books by Grafakos~\cite[Chap.~6]{grafakos:modern},
Runst and Sickel~\cite[Chap.~2]{RS},
and Taylor~\cite[Chap.~4 and~13]{taylor}.

For some space $X\subset C^{1}\rzd$ we will denote by $X_\mathrm{ir}$ the (open) subspace consisting of all
\underline injective (embedded) and \underline regular curves in~$X$.

The standard scalar product in~$\R^n$ is denoted by $\sp{\cdot,\cdot}$, for complex vectors $a,b\in\C^n$ we define
$\sp{a,b}_{\C^n}:=\sum_{k=1}^da_k\overline{b_k}$. The $L^2$-scalar product is, as usual, given by
$\sp{f,g}_{L^2}:=\int_0^1\sp{f(u),g(u)}\d u$.

Unless stated otherwise, we will assume
\[ \a\in(2,3) \]
throughout this paper.

\paragraph{Acknowledgements.}
The first author was supported by the Swiss National Science Foundation Grant Nr.~200020\_125127 and
the Leverhulm trust.
The second author was supported by DFG Transregional Collaborative Research Centre SFB~TR~71.
\section{Continuous differentiability}\label{sect:differentiability}

In this section, we want to prove the following proposition from which Theorem~\ref{thm:FrechetDifferential}
will follow quite easily. Recall that
$U_\varepsilon = (\mathbb R / \mathbb Z )\times ([-1/2, - \varepsilon] \cup [\varepsilon, 1/2])$.
\newcommand{\gt} {\gamma + \tau h}

\begin{proposition}\label{prop:formula}
  For $\alpha \in (2,3)$ the energies $\E$ are continuously differentiable
  on $\Wir$. The derivative of $\E$ at $\gamma \in \Wir$ in direction $h \in \W$
  is given by
  \begin{equation}\label{eq:formulaD}
  \begin{split}
    &\delta\E(\g;h)  = \\
    &\lim_{\varepsilon \searrow 0} \iint\limits_{U_\eps} \Bigg\{
    2\left(
      \frac 1 {\abs{\D\g}^\a}-\frac 1 {d_\g(u+w,u)^\a} \right)
      \left\langle \frac{\dg(u)}{\abs{\dg(u)}^2},\dh(u) \right\rangle 
\\
    &\qquad -\a 
      \left( \frac{\sp{\D\g,\D h}}{\abs{\D\g}^{\a+2}}
	-\frac{ \left.\tfrac{\d}{\d\tau}\right|_{\tau=0}d_{\gt}(u+w,u)}
		{d_\g(u+w,u)^{\a+1}} \right)\Bigg\} 
\\ &\hspace{7.6cm} |\dg(u+w)| |\dg(u)|\d w\d u .
  \end{split}
  \end{equation}
\end{proposition}

Note that since $\gamma \in C^1$ the derivative 
$\left.\tfrac{\d}{\d\tau}\right|_{\tau=0}d_{\gt}(u+w,u)$ is well defined for almost all
$(u,w) \in \mathbb R / \mathbb Z \times [-1/2, 1/2]$. 
More precisely, we can deduce from
\begin{equation*}
 d_\g(u+w,u)= \min\left\{ \mathscr L(\gamma|_{[u,u+w]}),  \mathscr L (\gamma) - \mathscr L(\gamma|_{[u,u+w]})\right\}
\end{equation*}
that 
\begin{multline}\label{eq:dd} 
 \left.\tfrac{\d}{\d\tau}\right|_{\tau=0} d_{\gt} (u+w,u) \\=
 \begin{cases} |w|\int_{0}^1 
    \left\langle \frac {\gamma'(u+\sigma w)}{|\gamma'(u+\sigma w)|} , h'(u+\sigma w) \right\rangle d\sigma,
    &\quad\text{ if } \length(\gamma|_{[u,u+w]}) <\frac 1 2 \length(\gamma),\\
    -|w|\int_{0}^1
    \left\langle  \frac {\gamma'(u+\sigma w)}{|\gamma'(u+\sigma w)|} , h'(u+\sigma w) \right\rangle d\sigma,
    &\quad\text{ if } \length(\gamma|_{[u,u+w]}) > \frac 1 2 \length(\gamma).
    \end{cases}
\end{multline}

To prove Proposition~\ref{prop:formula}, we will first show that the following approximations
of the energy $\E$, in which we cut off the singular part, are continuously differentiable 
and give a formula for the derivative.
For $\varepsilon >0$ we set
\begin{equation*}
  \Ee(\g) := \iint\limits_{ U_\eps} 
    \br{\frac1{\abs{\D\g}^\a}-\frac1{d_\g(u+w,u)^\a}}\abs{\dg(u+w)}\abs{\dg(u)}
    \d w\d u.  
\end{equation*}
To be more precise, we will show that $\Ee$ is $C^1$ on the space of all embedded 
regular curves of class $C^1$, which due to the embedding~\eqref{eq:embedding} 
especially implies the continuous differentiability on $\Wir$.

The general strategy of the proof will be fairly standard. We first derive
a formula for the pointwise variation of the integrand in the definition of $\Ee$
and $\E$ which holds almost everywhere. After that we will carefully analyse this formula
in order to prove that the integrand defines a continuously differentiable map
from $C_{\mathrm{ir}}^{1}(\mathbb R / \mathbb Z, \mathbb R^n)$ to $L^1 (U_{\varepsilon})$.
This allows us to deduce that $\Ee$ is continuously differentiable.

\begin{lemma} \label{lem:FrechetDifferentiable}
The functional $\Ee$ is continuously differentiable on the space
of all injective regular curves in $C^1(\mathbb R / \mathbb Z,\mathbb R^n)$.
The directional derivative
at $\gamma $ in direction $h\in C^1(\mathbb R / \mathbb Z, \mathbb R^n)$ is given by
\begin{equation}\label{eq:FrechetDifferentiable}
\begin{split}
    &\delta\Ee(\g;h)  = \\
    & \iint\limits_{U_\eps} \Bigg\{
    2\left(
      \frac 1 {\abs{\D\g}^\a}-\frac 1 {d_\g(u+w,u)^\a} \right)
      \left\langle \frac{\dg(u)}{\abs{\dg(u)}^2},\dh(u) \right\rangle 
\\
    &\qquad -\a 
      \left( \frac{\sp{\D\g,\D h}}{\abs{\D\g}^{\a+2}}
	-\frac{ \left.\tfrac{\d}{\d\tau}\right|_{\tau=0}d_{\gt}(u+w,u)}
		{d_\g(u+w,u)^{\a+1}} \right)\Bigg\} 
\\ &\hspace{7.6cm} |\dg(u+w)| |\dg(u)|\d w\d u .
  \end{split}
\end{equation}
\end{lemma}


\begin{proof}
Let $\gamma_0 \in C^1(\mathbb R / \mathbb Z, \mathbb R^n)$ be injective and regular and 
$U \subset C^1(\mathbb R / \mathbb Z, \mathbb R^n)$ be an open neighbourhood of $\gamma_0$ such that
there is a constant $c >0$ with
\begin{equation} \label{eq:fundamentalEstimates}
 \min\{|\gamma(u+w) - \gamma(u)|, d_\gamma(u+w,u)\}  \geq c |w|, \quad \quad
 \quad|\gamma'(u)| \geq c
\end{equation}
for all $\gamma \in U$ and $(u,w) \in \mathbb R / \mathbb Z \times [-1/2,1/2]$.

We will show that the integrand used to define the energies $\E$ and $\Ee$, i.~e.\@
\begin{equation*}
 (I\gamma) (u,w) := \left(\frac 1 {|\gamma(u+w) - \g(u)|^\alpha} - \frac 1 {d_\gamma(u+w,u)^\alpha} \right) 
  |\gamma'(u+w)| |\gamma'(u)|,
\end{equation*}
defines a continuously differentiable operator from $U$ into
$L^1(U_\varepsilon)$  for any $\eps>0$ with directional 
derivative
\begin{equation}\label{eq:derivativeOfIntegrand}
\begin{split}
 &\left. \frac {\d}{\d\tau} \left(I(\gamma + \tau h) \right)(u,w) \right|_{\tau =0} \\
  &=
 \left(\frac 1 {|\gamma(u+w) - \g(u)|^\alpha} - \frac 1 {d_\gamma(u+w,u)^\alpha} \right)  
\\
& \qquad \qquad \cdot 
  \left( \left\langle \frac{\gamma'(u)}{|\gamma'(u)|}, h'(u) \right\rangle\abs{\dg(u+w)} +  
  \left\langle \frac{\gamma'(u+w)}{|\gamma'(u+w)|}, h'(u+w) \right\rangle\abs{\dg(u)}\right) \\
  & - \alpha  \left(\frac{\left\langle \D\g ,\D h \right\rangle }{|\D\g|^{\a+2}}
    -\frac{ \left.\frac \d {\d\tau}\right|_{\tau=0}d_{\gt}(u+w,u)}{d_\g(u+w,u)^{\a+1}}\right)
  \abs{\dg(u+w)}\abs{\dg(u)}.
\end{split}
\end{equation}

The statement then follows from the chain rule and the fact that the operator
\begin{gather*}
 L^1(U_\varepsilon) \to \mathbb R, \\
 g \mapsto \iint_{U_\varepsilon} g(u,w) \d u\d w,
\end{gather*}
is continuously differentiable as it is a bounded linear operator.
 
The only non-trivial thing here is to deal with the intrinsic distance $d_{\gamma}$ in the integrand 
that defines $\Ee$. Obviously $d_\gamma(u,w)$ defines a continuous 
operator from $C^1 (\mathbb R / \mathbb Z, \mathbb R^n)$ to 
$L^\infty(\mathbb R / \mathbb Z \times [-1/2, 1/2])$.

Using the fact that one has
\begin{equation*}
 d_\g(u+w,u)= \min\left\{ \mathscr L(\gamma|_{[u,u+w]}),  \mathscr L (\gamma) - \mathscr L(\gamma|_{[u,u+w]})\right\}
\end{equation*}
and that $\gamma$ is regular, one can see that 
\begin{equation}\label{eq:dd=D}
 \left.\tfrac{\d}{\d\tau}\right|_{\tau=0} d_{\gt} (u+w,u) = D(\gamma,h)(u,w)
\end{equation}
for all $u,w$ with $\mathscr L(\gamma|_{[u,u+w]}) \not= \frac 12 {\mathscr L (\gamma)}$ where
\begin{equation}\label{eq:Dgh}
 D(\gamma;h)(u,w) := \begin{cases} |w|\int_{0}^1 
    \left\langle \frac {\gamma'(u+\sigma w)}{|\gamma'(u+\sigma w)|} , h'(u+\sigma w) \right\rangle d\sigma,
    &\quad\text{ if } \mathscr L(\gamma|_{[u,u+w]}) <\frac 1 2 \mathscr L(\gamma),\\
    -|w|\int_{0}^1
    \left\langle  \frac {\gamma'(u+\sigma w)}{|\gamma'(u+\sigma w)|} , h'(u+\sigma w) \right\rangle d\sigma,
    &\quad\text{ if } \mathscr L(\gamma|_{[u,u+w]}) \geq \frac 1 2 \mathscr L(\gamma).
    \end{cases}
\end{equation}

Since $\gamma$ is regular, the set $\mathscr L(\gamma|_{[u,u+w]}) = \frac 1 2 \mathscr L(\gamma)$
is a compact $C^1$ submanifold of 
$\mathbb R/ \mathbb Z\times [-1/2, 1/2]$ and hence a null set. Thus
\eqref{eq:dd=D} and \eqref{eq:derivativeOfIntegrand} hold almost everywhere.

Obviously, $D$ defines a continuous operator from the space
$C_{\mathrm{ir}}^1(\mathbb R / \mathbb Z, \mathbb R^n) \times C^1 (\mathbb R / \mathbb Z, \mathbb R^n)$
to $L^1(U_{\varepsilon})$.

From Equation~\eqref{eq:derivativeOfIntegrand} we can read off that 
\begin{equation*}
 \left(D I (\gamma)(h)\right) := \left. \frac {\d}{\d\tau} I(\gamma + \tau h)(u,w) \right|_{\tau =0}
\end{equation*}
defines an operator $C^1(\R/\Z,\R^n)\to L^1(U_\varepsilon)$ that
continuously depends on $\gamma$.
Hence $I$ is a continuously differentiable operator from $U$ to $L^1(U_\varepsilon)$.

Integrating and using a suitable reparametrization we then derive~\eqref{eq:FrechetDifferentiable}
from \eqref{eq:derivativeOfIntegrand}.
\end{proof} 

Unfortunately, the energies $\Ee$ do not form a Cauchy sequence
in $C^1(\Wir)$ -- actually not even in $C^0(\Wir)$ basically due to the fact that bounded sequences in $L^1$ are not uniformly integrable. 
We will deduce Proposition~\ref{prop:formula} from Lemma~\ref{lem:FrechetDifferentiable}
which roughly speaking shows that $\Ee$ is nearly a Cauchy sequence 
in $C^1(X_{\delta})$ for certain subsets $X_\delta \subset \W[(\alpha +1)/2,2]$, $\delta\ge0$. 
We will allow subsets $X_\delta  \subset \W[(\alpha +1)/2,2]$ which satisfy
the following substitute of the uniform integrability property
\begin{equation} \label{eq:uniformintegrability}
 \limsup_{\varepsilon \rightarrow 0}\sup_{\gamma \in X_\delta} 
  \left(\iint_{\mathbb R / \mathbb Z \times [-\eps, \eps]}
  \frac {|\gamma'(u+w) - \gamma'(u)|^2}{|w|^\alpha} dw du \right) ^{1/2}
\leq \delta.
\end{equation}

Recall that $\displaystyle\lip_Y E = \sup_{\substack{f,\tilde f\in Y\\f\ne\tilde f}}\frac{\abs{E(f)-E(\tilde f)}}{\norm{f-\tilde f}}$
for some functional $E$ with $Y\subset\textrm{domain } E$.

\begin{lemma} \label{lem:CauchySequence}
 Let $\gamma_0 \in \Wir$. Then there is an open neighborhood $U \subset \Wir$ of~$\g_0$
 and a constant $C < \infty$, such that 
 $\Ee$ satisfies
\begin{gather}
      \limsup_{\varepsilon_1, \varepsilon_2 \rightarrow 0}\lip_{U \cap X_\delta} (\E_{\eps_1} - \E_{\eps_2})
      \le C \delta\label{eq:CauchySequence2}
\end{gather}
for all subsets $X_\delta \subset \W$ satisfying \eqref{eq:uniformintegrability}
with $\delta \in[0,1]$.
\end{lemma}

For fixed $\gamma_0 \in \Wir$ and $h \in \W$ we will apply this lemma later to the sets
\begin{equation*}
 X_0 := \{ \gamma_0 + \tau h: \tau \in (-a,a)\}, \qquad 0< a < \infty,
\end{equation*}
and 
\begin{equation*}
 X_\delta := \{ \gamma \in \W : \|\gamma - \gamma_0\|_{\W} \leq \delta\},\qquad\delta>0.
\end{equation*}
Of course we have for $\gamma_\tau := \gamma +\tau h$, $|\tau| \leq a $,
\begin{align}
  \Bigg( \iint_{\mathbb R / \mathbb Z \times [-\eps, \eps]} &
  \frac {|\gamma_\tau'(u+w) - \gamma_\tau'(u)|^2}{|w|^\alpha} dw du \Bigg)^{1/2}  \nonumber \\
  &\leq \left( \iint_{\mathbb R / \mathbb Z \times [-\eps, \eps]}
  \frac {|\gamma'(u+w) - \gamma'(u)|^2}{|w|^\alpha} dw du \right)^{1/2}
  \nonumber\\ 
  & \quad +  a \left( \iint_{\mathbb R / \mathbb Z \times [-\eps, \eps]}
  \frac {|h'(u+w) - h'(u)|^2}{|w|^\alpha} dw du \right)^{1/2} 
  \nonumber\\
  &\rightarrow 0\label{eq:X0}
\end{align}
as $\varepsilon \rightarrow 0$ and
for $\gamma \in X_\delta$ we have
\begin{equation}
  \begin{aligned}
  \Bigg( \iint_{\mathbb R / \mathbb Z \times [-\eps, \eps]} &
  \frac {|\gamma'(u+w) - \gamma'(u)|^2}{|w|^\alpha} dw du \Bigg)^{1/2} \\
  &\leq \left( \iint_{\mathbb R / \mathbb Z \times [-\eps, \eps]}
  \frac {|\gamma_0'(u+w) - \gamma_0'(u)|^2}{|w|^\alpha} dw du \right)^{1/2}
 +  \delta \rightarrow \delta\label{eq:Xdelta}
\end{aligned}
\end{equation}
so both satisfy \eqref{eq:uniformintegrability}.

\begin{proof}

 Using that $\W$ continuously embeds into $C^1$ and $\gamma_0$ is a injective regular curve, we can find
 an open neighborhood $U \subset \Wir$  of $\gamma_0$ and 
 a constant $c >0$ such that~\eqref{eq:fundamentalEstimates} holds
for all $\gamma \in U$ and $(u,w) \in \mathbb R / \mathbb Z \times [-1/2,1/2]$. Making $U$ smaller if
necessary, we can also achieve that there is an $\varepsilon_0 >0$ such that
\begin{equation*}
 d_{\gamma}(u+w,u) = \mathscr L (\gamma|_{[u,u+w]})
\end{equation*}
for all $\gamma \in U$ and  $w \in [-\varepsilon_0, \varepsilon_0]$.
 \renewcommand{\gt}{\gamma}%
 Let now $\varepsilon_0 > \varepsilon_2 > \varepsilon_1$ and let us set 
 \[ F^{(\alpha)}:= \E_{\varepsilon_2} - \E_{\varepsilon_1}. \]
 We will now rewrite
 this difference in a more convenient form. For this let us 
 introduce the function
 \[ g^{(\alpha)}(\zeta,\eta,\vartheta,\iota) := \frac{\zeta^{-\alpha}-\eta^{-\alpha}}{\eta^2-\zeta^2}\vartheta\iota \]
 which is Lipschitz continuous and positive on $[\tilde c,\infty)^4$ for any $\tilde c>0$.
 We define for $u\in\R/\Z$, $w\in[-\eps,\eps]$
 \begin{align*}
  \G[\gt]:(u,w) &\mapsto g^{(\alpha)}\br{\abs{\int_0^1\dgt(u+\theta_1w)\d\theta_1},\int_0^1\abs{\dgt(u+\theta_2w)}\d\theta_2,\abs{\dgt(u+w)},\abs{\dgt(u)}}.
 \end{align*}
 We have chosen $U$ in such a way that the arguments in $\G[]$ are 
 uniformly bounded away from zero.

 We decompose the integrand in the definition of $\E$ for $|w| \leq \varepsilon_0$ into
 \begin{align*}
  &\br{\frac1{\abs{\D\gt}^\alpha}-\frac1{d_{\gt}(u+w,u)^\alpha}}\abs{\dgt(u+w)}\abs{\dgt(u)}\\
  &=\frac1{\abs w^\a}\br{\frac1{\abs{\int_0^1\dgt(u+\theta_1w)\d\theta_1}^\alpha}-\frac1{\br{\int_0^1\abs{\dgt(u+\theta_2w)}\d\theta_2}^\alpha}}\abs{\dgt(u+w)}\abs{\dgt(u)}\\
  &=\G[\gt](u,w)\frac{\br{\int_0^1 \abs{\dgt(u+\theta_2w)}\d\theta_2}^2 - \abs{\int_0^1\dgt(u+\theta_1w)\d\theta_1}^2}{\abs w^{\alpha}}  \\
  &=\G[\gt](u,w)\frac{\iint_{[0,1]^2} \br{\abs{\dgt(u+\theta_1w)}\abs{\dgt(u+\theta_2w)} - \sp{\dgt(u+\theta_1w),\dgt(u+\theta_2w)}} \d\theta_1\d\theta_2}{\abs w^{\alpha}}.
 \end{align*}
 Using $2\abs a\abs b-2\sp{a,b} = \abs{a-b}^2 - \abs{\abs a-\abs b}^2$ for $a,b\in\R^n$ 
 this can be written as
 \begin{align*}
  &\tfrac12 \G[\gt](u,w)\frac{\iint_{[0,1]^2}\abs{\dgt(u+\theta_1w)-\dgt(u+\theta_2w)}^2\d\theta_1\d\theta_2}{\abs w^\alpha} \\
  &\quad{}-\tfrac12 \G[\gt](u,w)\frac{\iint_{[0,1]^2}\br{\abs{\dgt(u+\theta_1w)}-\abs{\dgt(u+\theta_2w)}}^2\d\theta_1\d\theta_2}{\abs w^\alpha}.
 \end{align*}
 Hence, 
 \begin{align*}
  &F^{(\alpha)}(\gamma) \\ &=  \tfrac12 \int_{\mathbb R / \mathbb Z}
  \int_{\varepsilon_1 < |w| < \varepsilon_2}\G[\gt](u,w)\frac{\iint_{[0,1]^2}\abs{\dgt(u+\theta_1w)-\dgt(u+\theta_2w)}^2\d\theta_1\d\theta_2}{\abs w^\alpha}\d w\d u \\
  &\quad{}-\tfrac12  \int_{\mathbb R / \mathbb Z}
  \int_{\varepsilon_1 < |w| < \varepsilon_2}\G[\gt](u,w)\frac{\iint_{[0,1]^2}\br{\abs{\dgt(u+\theta_1w)}-\abs{\dgt(u+\theta_2w)}}^2\d\theta_1\d\theta_2}{\abs w^\alpha}\d w\d u \\
  &=: \tfrac12 F_1^{(\alpha)}(\gt) - \tfrac12 F_2^{(\alpha)}(\gt).
 \end{align*}

\renewcommand{\gt}{\tilde\gamma}%
 To estimate the difference $F^{(\alpha)}(\gt)-F^{(\alpha)}(\g)$, we first consider
 \begin{align*}
  &\abs{\G[\gt](u,w)-\G(u,w)} \\
  &\le C\abs{\abs{\int_0^1\dtg(u+\theta_1w)\d\theta_1}-\abs{\int_0^1\dg(u+\theta_2w)\d\theta_2}}
  \\
  & \quad{}+C\abs{\int_0^1\br{\abs{\dgt(u+\theta w)}-\abs{\dg(u+\theta w)}}\d\theta} \\
  &\quad{}+C\abs{\abs{\dgt(u+w)}-\abs{\dg(u+w)}} +C\abs{\abs{\dgt(u)}-\abs{\dg(u)}} \\
  &\le C\int_0^1\abs{\dgt(u+\theta w)-\dg(u+\theta w)}\d\theta +C\abs{{\dgt(u+w)}-{\dg(u+w)}} +C\abs{{\dgt(u)}-{\dg(u)}} \\
  &\le C\norm{\dgt-\dg}_{L^\infty}.
 \end{align*}
 We arrive at
 \begin{align*}
  &\abs{F_1^{(\alpha)}(\gt)-F^{(\alpha)}_1(\g)} \\
  &\le \iint\limits_{\R/\Z\times[-\eps_2,\eps_2]}\abs{\G[\gt](u,w)-\G(u,w)}\frac{\iint_{[0,1]^2}\abs{\dgt(u+\theta_1w)-\dgt(u+\theta_2w)}^2\d\theta_1\d\theta_2}{\abs w^\alpha}\d w\d u \\
  &\quad{}+\iint\limits_{\R/\Z\times[-\eps_2,\eps_2]}\abs{\G[\g](u,w)}\\
  & \qquad \qquad  \frac{\iint_{[0,1]^2}\abs{\abs{\dgt(u+\theta_1w)-\dgt(u+\theta_2w)}^2-\abs{\dg(u+\theta_1w)-\dg(u+\theta_2w)}^2}\d\theta_1\d\theta_2}{\abs w^\alpha} \\
  & \hspace{11.2cm}\d w\d u \\
  &\le C\norm{\dgt-\dg}_{L^\infty}\iint\limits_{[0,1]^2}\iint\limits_{\R/\Z\times[-\eps_2,\eps_2]}\frac{\abs{\dgt(u+\theta_1w)-\dgt(u+\theta_2w)}^2}{\abs w^\alpha}\d w\d u\d\theta_1\d\theta_2 \\
  &\quad{}+C\iint\limits_{[0,1]^2}\iint\limits_{\R/\Z\times[-\eps_2,\eps_2]} \\
  & \qquad  \frac{\abs{(\dgt+\dg)(u+\theta_1w)-(\dgt+\dg)(u+\theta_2w)}\abs{(\dgt-\dg)(u+\theta_1w)-(\dgt-\dg)(u+\theta_2w)}}{\abs w^\alpha} \\
  & \hspace{10.2cm}\d w\d u\d\theta_1\d\theta_2 \\
  &\le C\seminorm{\dgt}_{\W[(\a-1)/2](\R/\Z\times[-2\eps_2,2\eps_2])}^2\norm{\dgt-\dg}_{L^\infty}  \\
  &\quad + C \seminorm{\dgt+\dg}_{\W[(\a-1)/2](\R/\Z\times[-2\eps_2,2\eps_2])}\seminorm{\dgt-\dg}_{\W[(\a-1)/2](\R/\Z\times[-2\eps_2,2\eps_2])}
 \end{align*}
 where we set for a subset $S\subset\R/\Z\times[-\frac12,\frac12]$
 \begin{equation*}
  \seminorm{f}_{\W[(\a-1)/2](S)}
  := \left( \iint_{S} \frac{\abs{f(u+w)-f(u)}^2}{\abs w^{\a}} \d w\d u\right)^{1/2}.
 \end{equation*}

 For the second term we compute
 \begin{align*}
  &\abs{F_2^{(\alpha)}(\gt)-F_2^{(\alpha)}(\g)} \\
  &\le \! \! \iint\limits_{\R/\Z\times[-\eps_2,\eps_2]} \! \! \! \!\abs{\G[\gt](u,w)-\G(u,w)}\frac{\iint_{[0,1]^2} \!\! \br{\abs{\dgt(u+\theta_1w)}-\abs{\dgt(u+\theta_2w)}}^2\d\theta_1\d\theta_2}{\abs w^\alpha}\d w\d u \\
  &\qquad{}+\iint\limits_{\R/\Z\times[-\eps_2,\eps_2]}\abs{\G[\gt](u,w)} \cdot\\
  &\qquad \cdot\frac{\iint_{[0,1]^2}\abs{\br{\abs{\dgt(u+\theta_1w)}-\abs{\dgt(u+\theta_2w)}}^2-\br{\abs{\dg(u+\theta_1w)}-\abs{\dg(u+\theta_2w)}}^2}\d\theta_1\d\theta_2}{\abs w^\alpha}\\
  & \hspace{11.2cm} \d w\d u \\
  &\le C\norm{\dgt-\dg}_{L^\infty}\iint\limits_{[0,1]^2}\iint\limits_{\R/\Z\times[-\eps_2,\eps_2]}\frac{\abs{\dgt(u+\theta_1w)-\dgt(u+\theta_2w)}^2}{\abs w^\alpha}\d w\d u\d\theta_1\d\theta_2 \\
  &\quad{}+C\iint\limits_{[0,1]^2}\iint\limits_{\R/\Z\times[-\eps_2,\eps_2]}\frac{\abs{\br{\abs{\dgt(u+\theta_1w)}-\abs{\dgt(u+\theta_2w)}}+\br{\abs{\dg(u+\theta_1w)}-\abs{\dg(u+\theta_2w)}}}}{\abs w^{\alpha/2}}\cdot \\
  &\qquad\quad{}\cdot\frac{\abs{\br{\abs{\dgt(u+\theta_1w)}-\abs{\dgt(u+\theta_2w)}}-\br{\abs{\dg(u+\theta_1w)}-\abs{\dg(u+\theta_2w)}}}}{\abs w^{\alpha/2}}\d w\d u\d\theta_1\d\theta_2 \\
  &\le C\norm{\dgt-\dg}_{L^\infty}\seminorm{\dgt}_{\W[(\a-1)/2](\R/\Z\times[-2\eps_2,2\eps_2])}^2 \\
  &\qquad + C\seminorm{\abs\dgt+\abs\dg}_{\W[(\a-1)/2](\R/\Z\times[-2\eps_2,2\eps_2])}\seminorm{\abs\dgt-\abs\dg}_{\W[(\a-1)/2](\R/\Z\times[-2\eps_2,2\eps_2])}.
 \end{align*}

 Using the chain and product rule for Sobolev spaces and the formula
 \[
  \abs\dgt-\abs\dg 
  = \frac{\sp{\dgt+\dg,\dgt-\dg}}{\abs\dgt+\abs\dg},
 \]
 we get
 \begin{align*}
  \seminorm{\abs\dgt-\abs\dg}_{\W[(\a-1)/2](\R/\Z\times[-2\eps_2,2\eps_2])} 
  & \le \seminorm{\abs\dgt-\abs\dg}_{\W[(\a-1)/2,2](\R/\Z)} \\
  &  \le C \norm{\dgt-\dg}_{\W[(\a-1)/2,2](\R/\Z,\R^n)}
  \end{align*}
 and hence
 \begin{align*}
    \abs{F_2^{(\alpha)}(\gt)-F_2^{(\alpha)}(\g)}
   &\le C \left(\seminorm{\dgt}_{\W[(\a-1)/2](\R/\Z\times[-2\eps_2,2\eps_2])}^2 +
    \seminorm{\abs\dgt+\abs\dg}_{\W[(\a-1)/2](\R/\Z\times[-2\eps_2,2\eps_2])} \right)
   \\
    &\|\dgt-\dg\|_{\W[(\a-1)/2,2](\R/\Z,\R^n)}.
 \end{align*}
 From this the claim follows.
\end{proof}

\begin{proof}[Proposition~\ref{prop:formula}]
 From the classification of all embedded regular curves of finite energy in \cite{blatt:bre}
 we get $\E(\g_0) < \infty$ for all $\g_0 \in \Wir$.
 From this we deduce immediately that $\Ee$ converges to $\E$ pointwise as 
 $\eps$ tends to $0$.

 We begin by proving that \emph{directional derivatives} exist for all directions $h\in\W$.
 Let us fix $\gamma_0 \in \Wir$ and let $U \subset \Wir$ and $C< \infty$ be as in 
 Lemma \ref{lem:CauchySequence}.
 Applying first Lemma~\ref{lem:CauchySequence} with $X_0=\{\gamma_0 + \tau h: \tau \in (-\tau_0,\tau_0)\}$
 for $\tau_0$ small enough, we deduce for
 \[ f_\eps : \tau\mapsto \Ee(\g_0+\tau h) \]
 that
 \begin{align}
  &\abs{f_{\eps_1}'(\tau)-f_{\eps_2}'(\tau)}
  =\abs{\delta\Ee[\eps_1](\g_0+\tau h;h)-\delta\Ee[\eps_2](\g_0+\tau h;h)} \nonumber\\
  &\le\limsup_{\theta\to0}\abs{\frac{\Ee[\eps_1](\g_0+(\tau+\theta)h)-\Ee[\eps_1](\g_0+\tau h)}{\theta}
   - \frac{\Ee[\eps_2](\g_0+(\tau+\theta)h)-\Ee[\eps_2](\g_0+\tau h)}{\theta}} \nonumber\\  
  &\le \lip_{U\cap X_0}\br{\Ee[\eps_1]-\Ee[\eps_2]}\norm h_{\W}\label{eq:f-f}\\
  &\xrightarrow{\eps_1,\eps_2\searrow0}0\qquad\text{by \eqref{eq:X0}.}\nonumber
 \end{align}
 As $\Ee\to\E$ pointwise this proves that $\seq{f_\eps}$
 is a Cauchy sequence in $C^1((-\tau_0, \tau_0))$ converging to $\E(\gamma_0 + \tau h)=\lim_{\eps\searrow0}\Ee(\g_0+\tau h)$ as 
 $ \varepsilon \rightarrow 0$.
 Hence especially all directional derivatives of $\E$ exist and
 \begin{equation*}
  \delta \E (\gamma_0;h) = \lim_{\varepsilon \searrow 0} \delta \Ee (\gamma_0;h)
 \end{equation*}
 for all $\gamma_0 \in \Wir, h \in \W$.

 The next step is to establish \emph{G\^ateaux differentiability}. To this end we merely have to show
 $\delta\E(\g_0,\cdot)\in\br{\W}^*$ for $\g_0\in\Wir$. Linearity carries over from $\Ee$. For boundedness we choose
 $\delta\in(0,1]$ such that
 \begin{equation*}
  X_\delta := \{ \gamma \in \Wir: \|\gamma- \gamma_0\| \leq \delta\} \subset U.
 \end{equation*}
 Now
 \begin{align*}
  \delta \E(\gamma_0;h) 
  &= \delta \Ee(\gamma_0;h) + \delta \E(\gamma_0;h) - \delta \Ee(\gamma_0;h)
  \\
  &= \delta \Ee(\gamma_0;h) + \lim_{\varepsilon_1 \rightarrow 0} 
   (\delta \E_{\varepsilon_1} (\gamma_0;h) - \delta \Ee(\gamma_0;h))
 \end{align*}
 and thus, arguing as in~\eqref{eq:f-f} and recalling $\delta\Ee(\g_0;\cdot)\in\br{\W}^*$,
\begin{equation*}
\begin{aligned}
  | \delta \E(\gamma_0;h) |& \leq | \delta \Ee(\gamma_0;h) |
  +  \underbrace{\limsup_{\eps_1\searrow0}\lip_{U\cap X_\delta}\br{\Ee[\eps_1]-\Ee[\eps]}}_{<\infty} \|h\|_{\W}
  \\ & \leq C \|h\|_{\W}
\end{aligned}
\end{equation*}
 for all $\gamma_0 \in \Wir$ and $h \in \W$. Hence, $\E$ is G\^ateaux differentiable  and the 
differential $\left(\E\right)'(\gamma_0) \in (\W)^\ast$ is given by
\begin{equation*}
 \left(\E\right)'(\gamma_0) = \delta \E(\gamma_0;\cdot)
\end{equation*}
for all $\gamma_0 \in \Wir, h \in \W$.

Finally, to see that the differential is \emph{continuous}, let $\sigma >0$ be given and let us choose
$\delta>0$ and $\varepsilon>0$ so small that
\begin{equation*}
 \lip_{U\cap X_\delta}\br{\Ee[\eps_1]-\Ee[\eps_2]} \refeq[\le]{CauchySequence2} C \delta \leq \tfrac13 \sigma
\end{equation*}
for all $\varepsilon_1, \varepsilon_2 < \eps$. Then we
have for $\gamma \in X_\delta \cap U$ and any $h\in\W$
\begin{align*}
 |\delta E(\gamma;h) - \delta E(\gamma_0;h)|
 &\leq |\delta \E(\gamma;h) - \delta \Ee(\gamma;h)| + 
  |\delta \Ee(\gamma;h) - \delta \Ee(\gamma_0;h)| 
 \\ 
  &\quad + |\delta \Ee(\gamma_0;h) - \delta \E(\gamma_0;h)|
 \\
  &\refeq[\le]{f-f}  |\delta \Ee(\gamma;h) - \delta \Ee(\gamma_0;h)| +  \tfrac23 \sigma \|h\|_{\W}.
\end{align*} 
Since $\Ee$ is $C^1$ we deduce that there is an open neighborhood $V\subset X_\delta$
of $\gamma_0$ such that 
\begin{equation*}
 |\delta \Ee(\gamma;h) - \delta \Ee(\gamma_0;h)| \leq \tfrac13\sigma \|h\|_{\W}
\end{equation*}
and hence
\begin{equation*}
  |\delta \E(\gamma;h) - \delta \E(\gamma_0;h)| \leq \sigma \|h\|_{\W}.
\end{equation*}
This proves that $\left(\E \right)'$ is continuous from $\Wir$ into $\left( \W \right)^\ast$ and
hence $\E$ is $C^1(\Wir)$.
\end{proof}

\begin{proof}[Theorem~\ref{thm:FrechetDifferential}]
 The only thing left to do is to show that for curves $\gamma \in \Wir$ parametrized by arc-length
 and $h \in \W$ the derivative can be given in the form stated in the theorem.
 Using that $\gamma$ is parametrized by arc-length, we get from Proposition~\ref{prop:formula} 
 and~\eqref{eq:dd} that
 \begin{align*}
  \delta\E(\gamma;h) &\xleftarrow{\varepsilon \searrow 0} \iint_{U_\varepsilon} 2
\left(\frac 1 {|\gamma(u+w) - \g(u)|^\alpha} - \frac 1 {|w|^\alpha} \right) 
  \left\langle \gamma'(u), h'(u) \right\rangle  \\
  & - \alpha  \left(\frac{\left\langle \D\g ,\D h \right\rangle }{|\D\g|^{\a+2}}
    -\frac{D(\gamma,h)(u,w)}{|w|^{\a+1}}\right)
  dw du 
 \end{align*}
 where now 
 \begin{equation*}
  D(\gamma,h)(u,w) = |w|\int_0^1 \left\langle \gamma'(u+ \theta w), h'(u+\theta w) \right\rangle d\theta
 \end{equation*}
 for all $(u,w) \in \mathbb R / \mathbb Z \times (-1/2, 1/2)$. Hence,
 \begin{align*}
  &\delta\E(\gamma;h) 
  \xleftarrow{\varepsilon \searrow 0}\iint_{U_\varepsilon}  \Bigg\{2
    \left(\frac 1 {|\gamma(u+w) - \g(u)|^\alpha} - \frac 1 {|w|^\alpha} \right) 
    \left\langle \gamma'(u), h'(u) \right\rangle \\
    &\quad{} - \alpha  \left(\frac{\left\langle \D\g ,\D h \right\rangle }{|\D\g|^{\a+2}}
    -\frac{\int_0^1 \left\langle \gamma'(u+\theta w ), h'(u+\theta w) \right\rangle d\theta}
    {|w|^{\a}}\right) \Bigg\} dw du \\
  & = \iint_{U_\varepsilon} \Bigg\{ 2
    \left(\frac 1 {|\gamma(u+w) - \g(u)|^\alpha} - \frac 1 {|w|^\alpha} \right) 
    \left\langle \gamma'(u), h'(u) \right\rangle \\
    &\qquad\qquad{} - \alpha  \left(\frac{\left\langle \D\g ,\D h \right\rangle }{|\D\g|^{\a+2}}
    -\frac{\left\langle \gamma'(u), h'(u) \right\rangle}{|w|^{\a}}\right) \Bigg\}dw du\\
& = \iint_{U_\varepsilon} 
  \bigg((\alpha -2) \frac {\left\langle \gamma'(u), h'(u) \right\rangle} {|w|^\alpha} 
  + 2 \frac  {\left\langle \gamma'(u), h'(u) \right\rangle} {|\gamma(u+w) - \g(u)|^\alpha}
    \\
  &\qquad\qquad{} - \alpha  \frac{\left\langle \D\g ,\D h \right\rangle }{|\D\g|^{\a+2}} \bigg) dw du.
\end{align*}
\end{proof}
\section{Regularity of stationary points}
\label{sect:bootstrap}

In this section we prove Theorem~\ref{thm:rsp} so we are
looking at embedded curves $\gamma \in \W[(1+\alpha)/2]$ parametrized by arc-length that satisfy
\begin{equation}\label{eq:ele}
 \delta\E(\gamma;h) + \lambda \int_{\mathbb R / \mathbb Z} \sp{\gamma',h'} =0 \quad \forall h \in \W[(1+\alpha)/2]
  (\mathbb R / \mathbb Z, \mathbb R^n)
\end{equation}
where $\lambda>0$ and
\begin{align*}
 \delta\E(\gamma;h) = \lim_{\varepsilon \searrow 0}  \iint_{U_\eps} \Bigg( &
    (\a-2) \frac{\sp{\dg(u),\dh(u)}}{\abs w^\a}
    + 2 \frac{\sp{\dg(u),\dh(u)}}{\abs{\D\g}^\a} \\
    &{}-\a \frac{\sp{\D\g,\D h}}{\abs{\D\g}^{\a+2}}
    \Bigg) \d w\d u.
\end{align*}
To prove that $\gamma \in C^\infty(\mathbb R / \mathbb Z, \mathbb R^n)$, we first decompose 
\begin{equation}\label{eq:decomposE}
 \delta\E(\g;h) = \a\Q(\g,h) + \RR(\gamma,h)
\end{equation}
where
\begin{equation*}
 \Q (\gamma, h ) := \lim_{\varepsilon \searrow 0} 
  \iint_{U_\varepsilon} 
    \left(\frac {\left\langle \gamma'(u), h'(u) \right\rangle }{|w|^\alpha} 
      -\frac {\left\langle \gamma(u+w)-\gamma(u), h(u+w) - h(u) \right\rangle}{|w|^{\alpha+2}} \right) \d w\d u
\end{equation*}
and $\RR(\gamma,h)$ is given by
\begin{align*}
&\RR(\gamma,h) :=
  2 \lim_{\eps\searrow0}\iint_{U_\eps} \sp{\dg(u),\dh(u)} \br{\frac1{\abs{\D\g}^{\a}} - \frac1{\abs w^{\a}}} \d w\d u \\
   &{}-\a\lim_{\eps\searrow0}\iint_{U_\eps} \sp{\D\g,\D h} \br{\frac1{\abs{\D\g}^{\a+2}} - \frac1{\abs w^{\a+2}}} d wd u.
\end{align*}
Later on, it will become evident that, in contrast to~$\Q$, the integral defining~$\RR$ is not a principle value, i.~e.\@ we may
write~$U_0$ instead of~$U_\eps$.

It was already observed by He in \cite{he:elghf} and the second author in \cite{reiter:rkepdc} that
$\Q(\gamma,h)$ is a lower order perturbation of the $L^2$ product of 
$(-\Delta)^{\frac {\alpha +1} 4} \gamma$ and $(-\Delta)^{\frac {\alpha +1} 4} h$.
To see this, let us first extend $\Q$ to complex valued functions
by exchanging the scalar product on $\mathbb R^n$ to the scalar product on $\mathbb C^n$.
We denote by $\hat f(k)= \int_{\mathbb R / \mathbb Z} f(u) e^{-2\pi i ku} \d u$ the $k$-th Fourier coefficient of $f$.

\begin{proposition}[cf. \protect{\cite[Lemma 2.3]{he:elghf}, \cite[Proposition 1.4]{reiter:rkepdc}}] \label{prop:formOfQ}
 There is a sequence of real numbers $q_k$, $k \in \mathbb Z$, converging to a positive constant
 for $|k| \rightarrow \infty$
 such that for all $\gamma, h \in\W[\br{1+\alpha}/2 ,2](\mathbb R / \mathbb Z, \mathbb R^n)$
 we have
 \begin{equation} \label{eq:formOfQ}
  \Q(\gamma, h ) = \sum_{k\in \mathbb Z} q_k |k|^{\alpha +1}  \hat \gamma(k) \overline{\hat h (k)}.
 \end{equation}
\end{proposition}

Apart from this observation, the proof of Theorem~\ref{thm:rsp}  relies on the following estimate regarding the 
term $\RR(\gamma,h)$. Basically it lets us treat this term like a lower order perturbation.

\begin{proposition}\label{prop:RR}
 Let $\g\in\Wir[(\a+1)/2+\sigma]$ be parametrized by arc-length, $\sigma \geq 0$.
 \begin{enumerate}
  \item In the case $\sigma = 0$ we have $\RR(\gamma, \cdot)\in\br{H^{3/2+\eps ,2}}^*$ for any $\eps>0$.
  \item If $\sigma >0 $ we have $\RR(\gamma, \cdot)\in\br{H^{3/2- \hat \sigma,2}}^*$
  for all $\hat \sigma < \sigma$.	
 \end{enumerate}
\end{proposition}

We will prove  Proposition~\ref{prop:RR} using Sobolev embeddings and the fractional Leibniz rule 
for Bessel potential spaces (cf. Lemma~\ref{lem:product}).

First we will show, that the two summands building $\RR$ can be brought into a common form 
and can thus be dealt with simultaneously. For that we use the fundamental theorem of calculus
to get
\begin{align*}
 &\sp{\D\g,\D h} \br{\frac1{\abs{\D\g}^{\a+2}} - \frac1{\abs w^{\a+2}}} \\
 &= w^2\int_0^1 \int_0^1 \sp{\dg (u+s_1 w),\dh(u+s_2 w)} \br{\frac1{\abs{\D\g}^{\a+2}} - \frac1{\abs w^{\a+2}}}
  \d s_1 \d s_2.
\end{align*}
Furthermore, for $\beta>0$,
\begin{multline*}
 \frac 1 {|\gamma(u+w) - \gamma(u)|^\beta} - \frac 1 {|w|^\beta} 
  = \frac {|w|^\beta}{|\gamma(u+w)-\gamma(u)|^\beta} \cdot
     \frac {1- \frac {|\gamma(u+w)-\gamma(u)|^\beta}{|w|^\beta}} {|w|^\beta} \\
  = G^{(\beta)}\left(\frac {\gamma(u+w)-\gamma(u)}{w}\right)
     \frac {2-2\frac {|\gamma(u+w)-\gamma(u)|^2}{w^2}} {|w|^\beta} \\
  = \int_{0}^1 \int_{0}^1 G^{(\beta)}\left(\frac {\gamma(u+w)-\gamma(u)}{w}\right)
  \frac {|\gamma'(u+\tau_1 w) - \gamma'(u+\tau_2 w)|^2}{|w|^\beta} d\tau_1 d \tau_2
\end{multline*}
where
\begin{align*}
 G^{(\beta)}(z) := \frac 1 {2|z|^\beta} \cdot \frac {1- |z|^\beta}{1 - |z|^2}
\end{align*}
is an analytic function away from the origin. Defining
\[ g^{(\a,\beta)}_{s_1, \tau_1, \tau_2}(u,w) :=  G^{(\beta)} 
\left(\frac {\gamma(u+w) - \gamma(u)}{w} \right) \frac {|\gamma'(u+\tau_1 w) -\gamma'(u+\tau_2 w)|^2}{|w|^\alpha} 
\gamma'(u+s_1 w) \]
we thus get
%
%
\begin{equation}\label{eq:DecompositionOfR}
\begin{aligned}
 \RR(\gamma,h)
 &= \lim_{\eps \searrow 0} \Bigg\{2\iint\limits_{U_\eps}\iint\limits_{[0,1]^2} \left\langle g^{(\alpha,\alpha)}_{0,\tau_1,\tau_2}(u,w), 
  h'(u) \right\rangle \d\tau_1 \d\tau_2\d w\d u \\
 & \quad - \alpha \iint\limits_{U_\eps}\iiiint\limits_{[0,1]^4} \left\langle g^{(\alpha,\alpha+2)}_{s_1,\tau_1,\tau_2}(u,w), 
  h'(u+s_2 w) \right\rangle \d s_1 \d s_2 \d\tau_1 \d\tau_2\d w\d u \Bigg\}.
 \end{aligned}
\end{equation}

Thus using H\"older's inequality we get the estimate 
\begin{equation} \label{eq:estL1}
\begin{split}
 \RR(\gamma,h) \leq C \|h'\|_{L^\infty} \sup_{\substack{\beta\in\set{\a,\a+2},\\s_1,\tau_1,\tau_2\in[0,1]}}
 \int_{-1/2}^{1/2}\norm{g^{(\a,\beta)}_{s_1, \tau_1, \tau_2}(\cdot,w)}_{L^1}\d w
 \\ \leq  C \|h\|_{\W[3/2 + \varepsilon, 2]} \sup_{\substack{\beta\in\set{\a,\a+2},\\s_1,\tau_1,\tau_2\in[0,1]}}
 \int_{-1/2}^{1/2}\norm{g^{(\a,\beta)}_{s_1, \tau_1, \tau_2}(\cdot,w)}_{L^1}\d w
\end{split}
\end{equation}
for any $\eps>0$.

For $\s \in\R$ let
\begin{equation}\label{eq:Ds}
 D^\s := (-\Delta)^{\s/2}.
\end{equation}
By partial integration we infer for $\ts \in \mathbb R$
\[ \int\limits_{\R/\Z}
 \sp{g^{(\a,\beta)}_{s_1, \tau_1, \tau_2}(u,w),h'(u+s_2w)}
 \d u
 =\int\limits_{\R/\Z}
 \sp{D^\ts g^{(\a,\beta)}_{s_1, \tau_1, \tau_2}(u,w),D^{-\ts}h'(u+s_2w)}
 \d u \]
and we can estimate the absolute value by
\[ \norm{D^{-\ts}h'}_{\Lp[\infty]}\int\limits_{\R/\Z}
 \abs{D^\ts g^{(\a,\beta)}_{s_1, \tau_1, \tau_2}(u,w)}
 \d u
 \le C\norm{h}_{\W[3/2-\ts+\eps,2]}
 \norm{g^{(\a,\beta)}_{s_1, \tau_1, \tau_2}(\cdot,w)}_{\W[\ts,1]} \]
 for any $\eps>0$.
Combining this with Equation~\eqref{eq:DecompositionOfR} we get
\begin{equation*} 
 \abs{\RR(\g,h)}
 \le C\norm{h}_{\W[3/2-\ts+\eps,2]}\sup_{\substack{\beta\in\set{\a,\a+2},\\s_1,\tau_1,\tau_2\in[0,1]}}
 \int_{-1/2}^{1/2}\norm{g^{(\a,\beta)}_{s_1, \tau_1, \tau_2}(\cdot,w)}_{\W[\ts,1]}\d w
\end{equation*}
for all $\ts$ and $\varepsilon >0$.
To prove Proposition~\ref{prop:RR}, given $\sigma > \hat \sigma $
we set $\ts = (\s + \hat \s)/2 > \hat \s$  and $\varepsilon = \ts-\hat \s$ in the calculations above,
to get
\begin{equation} \label{eq:estHsigma}
 \abs{\RR(\g,h)}
 \le C\norm{h}_{\W[3/2-\hat \s,2]}\sup_{\substack{\beta\in\set{\a,\a+2},\\s_1,\tau_1,\tau_2\in[0,1]}}
 \int_{-1/2}^{1/2}\norm{g^{(\a,\beta)}_{s_1, \tau_1, \tau_2}(\cdot,w)}_{\W[\ts,1]}\d w.
\end{equation}
Proposition~\ref{prop:RR} now immediately follows from Estimate~\eqref{eq:estL1},
Estimate~\eqref{eq:estHsigma} and the succeding lemma.

\begin{lemma}\label{lem:basicEstimate}
 Let $\gamma \in\W[(\alpha+1)/2+\s,2] (\mathbb R / \mathbb Z, \mathbb R^n)$ with $\s\ge0$ and $\beta>0$.
\begin{enumerate}
 \item If $\s=0$ then $g^{(\a,\beta)}_{s_1, \tau_1, \tau_2} \in L^1(\R/\Z\times(-\tfrac12,\tfrac12),\R^n)$.
  Furthermore, there is a constant $C<\infty$
  independent of $\tau_1$, $\tau_2$, and $s_1$ such that
  \[ \norm{g^{(\a,\beta)}_{s_1, \tau_1, \tau_2}}_{\Lp[1]}\le C. \]
 \item If $\sigma> 0$ then $g^{(\a,\beta)}_{s_1, \tau_1, \tau_2,\cdot,\cdot} \in L^1((-\tfrac12,\tfrac12),H^{ \tilde \sigma,1}(\R/\Z,\R^n))$
 for all $\tilde \sigma < \sigma$
  and there is a constant $C<\infty$
  independent of $\tau_1$, $\tau_2$, and $s_1$ such that
  \[ \int_{-1/2}^{1/2}\norm{g^{(\a,\beta)}_{s_1, \tau_1, \tau_2}(\cdot,w)}_{\W[\tilde\sigma,1]}\d w \le C. \]
\end{enumerate}
\end{lemma}

\begin{proof}
 Let us first deal with the case $\sigma =0$. We get
 \begin{align*}
  &\|g^{(\a,\beta)}_{s_1, \tau_1, \tau_2}\|_{L^1 (\mathbb R / \mathbb Z\times(-\frac12,\frac12), \mathbb R^n)} \\
  &\leq 
  \int_{\mathbb R, \mathbb Z}  \int_{-1/2}^{1/2} G^{(\beta)}\left(\frac {\gamma(u+w) - \gamma(u)}{w} \right) 
  \frac{|\gamma'(u+\tau_1 w) - \gamma'(u+\tau_2 w)|^2}{|w|^\alpha} \abs{\gamma'(u+s_1 w)}  dw du 
  \\
  & \leq C \|\gamma'\|_{\Lp[\infty]} \int_{\mathbb R / \mathbb Z} \int_{-1/2}^{1/2}
     \frac{|\gamma'(u+\tau_1 w) - \gamma'(u+\tau_2 w)|^2}{|w|^\alpha}
    dw du
  \\
  & \leq C \|\gamma'\|_{\Lp[\infty]} \|\gamma\|_{\W[(\a+1)/2,2]}^2
 \end{align*}
  which proves the statement for $\sigma =0$.

 Since there is no suitable 
 product rule for $p=1$, we will
 estimate $\norm{g^{(\a,\beta)}_{s_1, \tau_1, \tau_2}(\cdot,w)}_{\W[\tilde \sigma,p]}$ for $p>1$ sufficiently small.
 For this we will use a small $\tilde p > p$ and let $q$
 be such that 
 \begin{equation*}
  \frac 1 p = \frac 1 {2\tilde p} +  \frac 1 {2\tilde p} + \frac 1 q + \frac 1 q,
 \end{equation*}
 i.~e.\@ we set
 \begin{equation*}
  q = 2 \frac {\tilde p p} {\tilde p -p}.
 \end{equation*}
Using that
\begin{equation*}
 \frac {\gamma(u+w)-\gamma(u)} {w} = \int_{0}^1 \gamma'(u+\tau w) d\tau,
\end{equation*}
that $\gamma$ is bi-Lipschitz, and that $G^{(\beta)}$ is analytic away from the origin, 
we get that
\begin{equation*}
 \left\|G^{(\beta)}\left(\frac {\gamma(\cdot +w)-\gamma(\cdot)}{w}\right)\right\|_{\W[\tilde\sigma,q]}
 \leq C \|\gamma\|_{\W[ \tilde \sigma +1,q]}
  \leq C
\end{equation*}
by the Sobolev embedding.

Using the fractional Leibniz rule (Lemma~\ref{lem:product}) three times, we derive for $\tilde\s\in(0,\s)$
\begin{align*}
  \|g^{(\a,\beta)}_{s_1, \tau_1, \tau_2}&(\cdot,w)\|_{\W[\tilde \sigma,p]} \\
 &\leq C 
 \left\|G^{(\beta)}\left(\frac {\gamma(\cdot +w)-\gamma(\cdot)}{w}\right) \right\|_{\W[ \tilde \sigma,q]}
 \|\gamma'\|_{\W[\tilde \sigma,q]}
\frac{\|\gamma'(\cdot + \tau_1 w) - \gamma'(\cdot + \tau_2 w )\|^2_{\W[\tilde\sigma,2\tilde p]}}{\abs w^\a}
 \\
 &\leq C \frac{\|\gamma'(\cdot + \tau_1 w) - \gamma'(\cdot + \tau_2 w )\|^2_{\W[\tilde\sigma,2\tilde p]}}{\abs w^\a}.
\end{align*}
We now choose $\tilde p>1$ so small that $H^{\sigma ,2}$ embeds into $H^{\tilde \sigma, 2 \tilde p}$
and hence
\begin{equation*}
 \|g^{(\a,\beta)}_{s_1, \tau_1, \tau_2}(\cdot, w)\|_{\W[\tilde \sigma,p]} 
  \leq  C \frac{\|\gamma'(\cdot + \tau_1 w) - \gamma'(\cdot + \tau_2 w )\|^2_{\W[\sigma,2]}}{\abs w^\a}.
\end{equation*}
Thus, recalling~\eqref{eq:Ds},
\begin{align*}
  \int_{-1/2}^{1/2}\|g^{(\a,\beta)}_{s_1, \tau_1, \tau_2}(\cdot,w)\|_{\W[\tilde \sigma,p]} \d w&\leq C
  \int_{-1/2}^{1/2} \frac{\|\gamma'(\cdot + \tau_1 w) - \gamma'(\cdot + \tau_2 w )\|^2_{\Lp[2]} }{\abs w^\alpha} dw
  \\
  & \quad + C
  \int_{-1/2}^{1/2} \frac {\|D^{\sigma+1}\gamma(\cdot + \tau_1 w) - D^{\sigma+1}\gamma(\cdot + \tau_2 w )\|^2_{\Lp[2]} 
   }{\abs w^\alpha} dw
  \\
&\leq C
\int_{-1/2}^{1/2} \frac{\|\gamma'(\cdot) - \gamma'(\cdot + (\tau_2 - \tau_1) w )\|^2_{\Lp[2]} }{\abs w^\alpha} dw
 \\
& \quad + C
  \int_{-1/2}^{1/2} \frac {\|D^{\sigma+1}\gamma(\cdot) 
- D^{\sigma+1} \gamma(\cdot + (\tau_2 - \tau_1) w )\|^2_{\Lp[2]} }{\abs w^\alpha} dw
\\
&\leq C |\tau_2 - \tau_1|
\int_{-1}^{1} \frac{\|\gamma'(\cdot) - \gamma'(\cdot +  w )\|^2_{\Lp[2]} }{\abs w^\alpha} dw
 \\
& \quad + C |\tau_2-\tau_1|
  \int_{-1}^{1} \frac {\|D^{\sigma+1}\gamma(\cdot) - D^{\sigma+1} \gamma(\cdot +  w )\|^2_{\Lp[2]} }{\abs w^\alpha} dw
\\
& \leq C \|\gamma'\|_{\W[(\a-1)/2,2]}^2 
 + C \|D^{\sigma+1}\gamma\|_{\W[(\a-1)/2,2]}^2 \leq C.
\end{align*}
This proves Lemma~\ref{lem:basicEstimate}.
\end{proof}

\begin{proof}[Theorem~\ref{thm:rsp}]
 Recall that any finite-energy curve belongs to $\W[(\a+1)/2,2]$ by~\cite{blatt:bre}.
 
 Let us assume that $\gamma \in H^{\frac {\alpha + 1} 2 + \sigma,2} (\mathbb R / \mathbb Z, \mathbb R^n)$
 for $\sigma \geq 0$ is a stationary point of the energy $\E + \lambda\length$.
 As the first variation of the length functional gives rise to a linear lower order term,
 Proposition~\ref{prop:formOfQ} also applies to
 \[ \tQ(\g,h) := \a\Q(\g,h)+\lambda\int_{\R/\Z}\sp{\dg,\dh}. \]
 In the case that $\sigma=0$ we get from the Euler-Lagrange Equation~\eqref{eq:ele} using the decomposition~\eqref{eq:decomposE} and Proposition~\ref{prop:RR} 
 \begin{equation*}
  \tQ( \gamma, \cdot)
  \in \left( H^{{3}/2+\eps,2} \right)^{\ast}
 \end{equation*}
 for any $\eps>0$.
 Using Proposition~\ref{prop:formOfQ} we hence get
 \begin{equation*}
   (q_k |k|^{\alpha+1-3/2-\eps} \hat \gamma (k))_{k\in \mathbb Z} \in \ell^2.
 \end{equation*}
 Together with the fact that $q_k$ converge to a positive constant as $k \rightarrow \infty$
 we get $ \left(|k|^{\alpha+1-3/2-\eps} \hat \gamma (k)\right)_{k \in \mathbb Z} \in \ell^2$
 and hence
  \begin{equation*}
    \gamma \in H^{\tfrac {\alpha +1 } 2 + \tfrac {\alpha -2 } 2 -\eps,2}
    (\mathbb R / \mathbb Z, \mathbb R^n).
  \end{equation*}
 For $\sigma >0$ we get using Proposition~\ref{prop:RR} 
 \begin{equation*}
  \tQ( \gamma, \cdot)
  \in \left( H^{\frac {3}2 - \hat \sigma } \right)^*\qquad\text{for all }\hat\sigma<\sigma
 \end{equation*} 
and arguing as above
\begin{equation*}
 \gamma \in H^{\br{\tfrac {\alpha +1 } 2 + \hat \s}+ \tfrac {\alpha -2 } 2,2}
  (\mathbb R / \mathbb Z, \mathbb R^n)
\end{equation*}
for all $\hat \s < \sigma$.

If we now initially assume that $\gamma \in H^{\frac {\alpha +1} 2 ,2}$ we deduce by induction
and since $\tfrac {\alpha-2} 2 >0$ that
 \begin{equation*}
  \gamma \in H^{s,2}(\mathbb R / \mathbb Z, \mathbb R^n)
 \end{equation*}
for all $s \in \mathbb R$ and thus $\gamma \in C^{\infty}(\mathbb R / \mathbb Z, \mathbb R^n)$.
This proves Theorem~\ref{thm:rsp}
\end{proof}

\begin{appendix}
 \section{Results on fractional Sobolev spaces}\label{sect:fractional}

Let us gather two results we used in the article: The product and
chain rule which go back to Coifman and Meyer~\cite{CM} and Christ and Weinstein~\cite{CW}.

\begin{lemma}[Leibniz Rule, cf.~\cite{CM}]\label{lem:product}
  Let $p_i,q_i,r \in (1,\infty)$, be such that $\frac 1{p_i} + \frac 1 {q_i} = \frac 1r$, for $i=1,2$ and 
  $s > 0$. 
  Then 
 \begin{equation*}
  \|f\cdot g\|_{\W[s,r]} \leq C \left(\|f\|_{\Lp[p_1]} \|g\|_{\W[s,q_1]} + \|f\|_{\W[s,p_2]} \|g\|_{\Lp[q_2]} \right).
 \end{equation*}
\end{lemma}

We also refer to Runst and Sickel~{\cite[Lem.~5.3.7/1~(i)]{RS}}. ---
For the following statement, one mainly has to treat $\norm{(D^k\psi)\circ f}_{\W[\s,p]}$ for $k\in\N\cup\set0$ and $\s\in(0,1)$
which is e.~g.\@ covered by~\cite[Thm.~5.3.6/1~(i)]{RS}.

\begin{lemma}[Chain rule, cf.~\cite{CM}]\label{lem:chain}
 Let $f\in H^{s,p}(\mathbb R / \mathbb Z, \mathbb R^n)$, $s > 0$, $p \in (1,\infty)$.
 If $\psi\in C^\infty(\R)$ such that $\psi$ and all its derivatives vanish at~$0$ then $\psi\circ f\in\W[s,p]$ and
 \[ \|\psi \circ f\|_{\W[s,p]} \le C\|\psi\|_{C^{\scriptstyle k}}\|f\|_{\W[s,p]} \]
 where $k$ is the smallest integer greater or equal to $s$.
\end{lemma}
\end{appendix}


\end{document}